\documentclass[12pt,reqno]{amsart}
\usepackage{amsmath, amssymb, amsthm} 
\usepackage{url}
\usepackage[breaklinks]{hyperref}
\usepackage{upgreek}
\usepackage{tikz}
\usepackage{mathrsfs}

\renewcommand{\(}{\left\(}
\renewcommand{\)}{\right\)}
\renewcommand{\[}{\left\[}
\renewcommand{\]}{\right\]}

\numberwithin{equation}{section}
 \theoremstyle{plain}
\newtheorem{theorem}{Theorem}[section]
\newtheorem{lemma}[theorem]{Lemma}

\newtheorem{conjecture}[theorem]{Conjecture}

\newtheorem{proposition}[theorem]{Proposition}

   \makeatletter
\def\proof{\@ifnextchar[{\@oproof}{\@nproof}}
\def\@oproof[#1][#2]{\trivlist\item[\hskip\labelsep\textit{#2 Proof of\
#1.}~]\ignorespaces}
\def\@nproof{\trivlist\item[\hskip\labelsep\textit{Proof.}~]\ignorespaces}

\makeatother

\begin{document}
\title[Zeros of Ramanujan-type Polynomials]{Zeros of Ramanujan-type Polynomials}


 \author{Bibekananda Maji}
\address{Bibekananda Maji\\ Department of Mathematics \\
Indian Institute of Technology Indore \\
Indore, Simrol, Madhya Pradesh 453552, India.} 
\email{bibek10iitb@gmail.com, bibekanandamaji@iiti.ac.in}

 \author{Tithi Sarkar}
\address{Tithi Sarkar\\ Department of Mathematics \\
Indian Institute of Technology Indore \\
Indore, Simrol, Madhya Pradesh 453552, India.} 
\email{tithijtr@gmail.com}



\maketitle

\begin{abstract}
Ramanujan's notebooks contain many elegant identities and one of the celebrated  identities is a formula for $\zeta(2k+1)$.  In 1972,  Grosswald gave an extension of the Ramanujan's formula for $\zeta(2k+1)$,  which contains a polynomial of degree $2k+2$.  This polynomial is now well-known as 
 the {\it{Ramanujan polynomial}} $R_{2k+1}(z)$, first studied by Gun,  Murty,  and Rath.  Around the same time,  Murty, Smith and Wang proved that all the non-real zeros of $R_{2k+1}(z)$ lie on the unit circle. Recently,  Chourasiya, Jamal, and the first author found a new polynomial while obtaining a Ramanujan-type formula for Dirichlet $L$-functions and named it as {\it{Ramanujan-type} polynomial} $R_{2k+1,p}(z)$.  In the same paper,  they conjectured that all the non-real zeros of $R_{2k+1,p}(z)$ lie on the circle $|z|=1/p$.  The main goal of this paper is to present a proof of this conjecture. 
\end{abstract}

\section{Introduction}
In the Second Notebook [33, p. 173, Ch. 14, Entry 21(i)] as well as in the Lost Notebook [34, p. ~319, Entry (28)],  Ramanujan mentioned the following remarkable identity for odd zeta values.  Considering $\alpha>0$ and $ \beta = \frac{\pi^2}{\alpha}$ and non-zero integer $k$,  he proved that
\begin{align}\label{1.2}
G_k(\alpha)=(-1)^kG_k(\beta)-2^{2k}\sum_{j=0}^{k+1}(-1)^{j-1}\frac{B_{2j}}{(2j)!}\frac{B_{2k+2-2j}}{(2k+2-2j)!}\alpha^{k+1-j}\beta^j,
\end{align}
where
$$
G_k(x)=x^{-k}\left\lbrace\frac{1}{2}\zeta(2k+1)+\sum_{n=0}^{\infty}\frac{n^{-2k-1}}{e^{2xn}-1}\right\rbrace.
$$
In 1972, Emil Grosswald \cite{item1} gave a notable  extension of Ramanujan's formula \eqref{1.2}.
Let us recall the divisor function $\sigma_{z}(m)=\sum_{d|m} d^{z}$ for any $z \in \mathbb{C}$.  For $\xi \in\mathbb{H}$,  $k \in \mathbb{Z}-\{ 0\}$,  we define
$$
\mathfrak{F}_k(\xi) :=\sum_{n=1}^{\infty}\sigma_{-k}(n)e^{2{\pi}in\xi}.
$$ 
Then we have
\begin{align}\label{1.3}
\mathfrak{F}_{2k+1}(\xi)-\xi^{2k}\mathfrak{F}_{2k+1}\left(-\frac{1}{\xi}\right)&=\frac{1}{2}\zeta(2k+1)(\xi^{2k}-1)\nonumber\\
&+\frac{(2{\pi}i)^{2k+1}}{2\xi}\sum_{j=0}^{k+1}\xi^{2k+2-2j}\frac{B_{2j}}{(2j)!}\frac{B_{2k+2-2j}}{(2k+2-2j)!}.                               
\end{align}
 Substituting $\xi=i \beta/\pi,$ $\alpha \beta={\pi}^2,$ with $\alpha,  \beta >0,$ the  identity \eqref{1.3} becomes same as the Ramanujan's identity \eqref{1.2} for $\zeta(2k+1).$ Later, Gun, Murty, and Rath \cite{item2} studied the finite sum involving Bernoulli numbers present on the right hand side of \eqref{1.3} and named it as the {\it Ramanujan polynomial}.  Thus,  for $k \in \mathbb{N}$ and $z \in \mathbb{C}$,  Ramanujan polynomials are defined as follows:
\begin{equation}\label{1.4}
R_{2k+1}(z):=\sum_{j=0}^{k+1}z^{2k+2-2j}\frac{B_{2j}}{(2j)!}\frac{B_{2k+2-2j}}{(2k+2-2j)!}.
\end{equation}
One can easily verify that Ramanujan polynomials are reciprocal or self-inverse.  In 2011,  Murty,  Smith,  and Wang \cite{MSW11} proved that all the non-real roots of $R_{2k+1}(z)$  lie on the unit circle $|z|=1$.   In 2013,  Lalin and Rogers \cite{LR13} observed similar behaviour for a few variant of Ramanujan polynomials involving Bernoulli numbers and Euler numbers.  One of the polynomials they \cite[Theorem 2.2]{LR13} have studied is the following: 
\begin{align}\label{Ramanujan-type polynomial_Lalin_Rogers}
R_{2k+1,  2}(z) = \sum_{j=1}^k(2^{2j}-1)(2^{2k+2-2j}-1)\frac{B_{2j}B_{2k+2-2j}}{{(2j)!(2k+2-2j)!}} z^{2k+2-2j}.
\end{align}
Interestingly, an  existence of the above polynomial can be seen from a Ramanujan-type formula studied by Berndt \cite[Theorem 4.7]{Berndt78} and Malurkar \cite{Malurkar25}. 
Recently,  Chourasiya,  Jamal,   and the first author \cite{CJM23} found the following Grosswald-type formula for the Dirichlet $L$-function $L(s,  \chi)$.  
Let $p$ be a prime number and $\chi_1$ be the principal Dirichlet character modulo $p$.  For any $k\in \mathbb{N}$ and $z \in \mathbb{H}$,  we define
\begin{align*}
 \mathfrak{F}_{k,\chi_1}(z):= \sum_{n=1}^\infty a_n \sigma_{-k,  \chi_1}(n) \exp( 2\pi i n z),  \quad \text{with} \quad  a_n = \begin{cases} 1, & \text{if}\,\, \gcd(n,p)=1,  \\
                    1-p,  & \text{if} \,\, \gcd(n,p)=p.
 \end{cases}
 \end{align*}
Then we have
\begin{align}
 \left( p z \right)^{2k} &  \mathfrak{F}_{2k+1,  \chi_1}\left(-\frac{1}{p^2 z} \right)  - \mathfrak{F}_{2k+1, \chi_1}(z) =   \frac{p-1}{2}L(2k+1,\chi_1)  \left\{ \left( pz \right)^{2k} -1  \right\}   \nonumber \\
   + & \frac{(2\pi i)^{2k+1}}{ 2z \, p^{2k+2}} \sum_{j=1}^{k}  {(p^{2j}-1)(p^{2k+2-2j}-1)\frac{B_{2j}}{(2j)!}\frac{B_{2k+2-2j}}{(2k+2-2j)!} } ( pz)^{2k+2-2j}.  \label{eqn_gen_grosswald_analogue}
\end{align}
The above identity motivated them to define Ramanujan-type polynomials as 
\begin{equation}\label{Ramanujan-type polynomial}
R_{2k+1,p}(z):= \sum_{j=1}^k(p^{2j}-1)(p^{2k+2-2j}-1)\frac{B_{2j}B_{2k+2-2j}}{{(2j)!(2k+2-2j)!}}(pz)^{2k+2-2j},  
\end{equation}
 where $k \in \mathbb{N}$ and $p$ is a prime number.  
 In the same paper,  they  \cite[Conjecture 2.9]{CJM23} have stated the below conjecture about the location of non-real roots of $R_{2k+1,p}(z)$.
\begin{conjecture}
Let $k \in \mathbb{N}$ and $p$ be any prime number, the only real root of Ramanujan-type polynomial $R_{2k+1,p}(z)$  is $z=0$ of multiplicity $2$.  Moreover,  the non-real roots are simple and lie on the circle $|z|=\frac{1}{p}$.
\end{conjecture}
At the end of their paper \cite[p.~748]{CJM23},  they also mentioned that the above conjecture might be true for any nature number $n \geq 2$ instead of any prime number $p$.  
The main goal of this article is to prove this conjecture.  Our main result is the following. 
\begin{theorem}\label{main theorem}
Let $k$ and $n \geq 2$ be two fixed natural numbers.  Then the following Ramanujan-type polynomial 
\begin{equation}\label{our polynomial}
R_{2k+1,n}(z):= \sum_{j=1}^k(n^{2j}-1)(n^{2k+2-2j}-1)\frac{B_{2j}B_{2k+2-2j}}{{(2j)!(2k+2-2j)!}}(nz)^{2k+2-2j},  
\end{equation}
has only one real root at $z=0$ of multiplicity $2$.  Again,  the non-real roots are simple and lie on the circle $|z|=\frac{1}{n}$.
\end{theorem}
 In the next section,  we collect a few well-known results about the location of zeros of reciprocal polynomials.
\section{Known results on reciprocal polynomials}
Let $p(x)\in \mathbb{R}[x]$ be a non-zero polynomial of degree $n$ with
$$
p(x)=a_0+a_1x+a_2x^2+\hdots+a_nx^n.
$$
We say that $p(x)$ is a reciprocal or self-inverse polynomial if and only if 
$$
x^np\left(\frac{1}{x}\right)=a_n+a_{n-1}x+a_{n-2}x^2+\hdots+a_0x^n=p(x).
$$
That implies that the coefficients of real reciprocal polynomials satisfy
$$
a_k=a_{n-k}, ~\forall~k=0,1,\cdots,n.
$$
\begin{theorem}[Lakatos \cite{item8}]\label{thm4.1}
If a reciprocal polynomial
$$
p(x)=\sum_{i=0}^{k}a_ix^i
$$ 
of degree $k$ with real coefficients satisfy the condition
\begin{equation}\label{4.1}
|a_k|\geq\sum_{i=0}^{k}|a_i-a_k|,
\end{equation}
then the polynomial $p(x)$  certainly bears all its zeros on the unit circle.
If the inequality \eqref{4.1} is strict, then the multiplicity of all the zeros of $p(x)$ is one.
\end{theorem}
Later, the above theorem has been modified by many mathematicians, thus the condition \eqref{4.1} for reciprocal polynomials has many improvements. Schinzel \cite{item10} gave one generalization of Theorem \ref{thm4.1} for any reciprocal polynomial over $\mathbb{C}$.
\begin{theorem}[Schinzel \cite{item10}]\label{thm4.2}
Let $p(z)$ be a reciprocal polynomial, say,
$$
p(z)=\sum_{j=0}^ka_jz^j
$$
of degree $k$ with $a_j \in \mathbb{C}, \forall j$. If the coefficients of $p(z)$ satisfy the condition
\begin{equation}\label{4.2}
|a_k|\geq \inf_{\substack{c, d\in \mathbb{C}\\
{|d|=1}}}\sum_{j=0}^{k}|ca_j-d^{k-j}a_k|,
\end{equation}
then all the zeros of $p(z)$ lie on $|z|=1$.  If the inequality \eqref{4.2} is strict,  then all the zeros of $p(x)$ must be simple. 
\end{theorem}
Next, we collect a few results related to Bernoulli numbers.  The following bounds  hold for Bernoulli numbers. 
\begin{lemma}\cite[p.~805]{AS64}  \label{bound_Bernoulli}
 For $n \in \mathbb{N}$,  we have
\begin{align}
\frac{2(2n)!}{(2\pi)^{2n}}  < | B_{2n}| < \frac{2 (2n)!}{(2\pi)^{2n}(1- 2^{1-2n})}.
\end{align}
\end{lemma}

We are now going to mention one useful identity about the convolution of Bernoulli polynomials. 
\begin{lemma}\cite[p.~31, ~eq. (3.2)]{item9}.   \label{lemma1.1}
For $m\in \mathbb{N}$,  we have
\begin{align*}
\sum_{j=0}^m  {{m}\choose{j}}B_j(a)B_{m-j}(b)=m(a+b+1)B_{m-1}(a+b)-(m-1)B_m(a+b).
\end{align*}
\end{lemma}
In the next section,  we present a proof of our main result.  
\section{Proof of Theorem \ref{main theorem}}

Before proving our main theorem,  we need following lemmas.   
\begin{lemma}\label{lemma3.2} Let $k \geq 1$ and $n \geq 2$ be two fixed natural numbers. 
Let $f_n(x)$ be a function $f_n:[0, ~2k-2]\rightarrow \mathbb{R}$ defined as 
\begin{align}\label{defn_f_p(x)}
f_n(x):=(n^{x+2}-1)(n^{2k-x}-1). 
\end{align}
The function $f_n(x)$ has global maxima at $x=k-1$ and 
 its minima at $ x=0,  2k-2$, with the same minimum value $(n^2-1)(n^{2k}-1)$ at both of the end points.
\end{lemma}

\begin{proof}
For fixed two natural numbers $n \geq 2$ and $k \geq 1$,  it is given that
$$
f_n(x):=(n^{x+2}-1)(n^{2k-x}-1). 
$$
Differentiating with respect to $x$,  one can see that
$$
f_n'(x)=-n^{x+2} \log(n)+n^{2k-x}\log(n)
\Rightarrow f_n''(x)=-\left( n^{x+2} +n^{2k-x} \right)\log^2(n).
$$
Further,  we can easily check  that $f_n''(x)$ is always negative in the given domain, so all the extrema points obtained by equating $f'_n(x)$ to zero will correspond to the local maximum values. Now we find extrema points of the function $f_n(x)$ by equating $f'_n(x)$ to zero, that is, 
$$f_n'(x)=0 \Rightarrow n^{x+2}=n^{2k-x} \Rightarrow
x=k-1.
$$
Thus, $x=k-1$ is the only local maxima point in the interval $[0,2k-2]$.  Hence it must global maxima as well.  Since the function is continuously differentiable, so we conclude that $x=2k-2 ~\text{and} ~x=0$ are the only points of minima.  Moreover,  we have
$$
f_n(0)=f_n(2k-2)=(n^2-1)(n^{2k}-1).
$$
Hence the result follows.
\end{proof}

\begin{lemma}\label{lemma3.3}
Let $k \geq 1$ and $n \geq 2$ be two fixed natural numbers.  The function $g_n:[0,2k-2]\rightarrow\mathbb{R}$ defined as 
\begin{equation}\label{defn_g(x)}
g_{n}(x):=\frac{(n^{x+2}-1)}{(2^{x+2}-1)}\frac{(n^{2k-x}-1)}{(2^{2k-x}-1)}, 
\end{equation}
has an upper bound $\frac{(n^{k+1}-1)^2}{3(2^{2k}-1)}.$
\end{lemma}

\begin{proof}
We can write $g_n(x)$ in terms of $f_n(x)$ as follows:
$$
g_n(x)=\frac{f_n(x)}{f_2(x)}.
$$
Therefore, the ratio of maximum value of $f_n(x)$ and the minimum value of $f_2(x)$ will provide us an upper bound for $g_n(x).$
From Lemma \ref{lemma3.2},  we know that the maxima of $f_n(x)$ is attained at $x=k-1$ and the corresponding maximum value is, 
\begin{align}\label{max}
f_n(k-1)=(n^{k+1}-1)^2.
\end{align}
Also the minimum value of $f_n(x)$ is attained at $x=0$ and in particular for $n=2$ the minimum value will be,
\begin{align}\label{min}
f_2(0)=(2^2-1)(2^{2k}-1)=3(2^{2k}-1).
\end{align}
Hence,  using \eqref{max} and \eqref{min},  it follows that
$$
g_n(x) \leq \frac{(n^{k+1}-1)^2}{3(2^{2k}-1)}.
$$
This completes the proof. 
\end{proof}
Next, we  state a result which will be crucial for the proof of our main theorem.    

\begin{proposition}\label{Schinzel's criteria for Rama_type}
Let $R_{2k+1, n}(z)$ be the Ramanujan-type polynomial defined as in \eqref{our polynomial} and $A_j$ be the $(2k-2-2j)^{th}$ coefficient of $z^{-2}R_{2k+1,n}\left(\frac{z}{n}\right).$ Then,  we have
$$
|A_{k-1}|\geq \inf_{\substack{c,d\in \mathbb{C}\\
{|d|=1}}}\sum_{j=0}^{k-1}|cA_j-dA_{k-1}|.  
$$
\end{proposition}

\begin{proof}

By appropriate shifting,  from the definition \eqref{our polynomial} of $R_{2k+1,n}(z)$,  we can write  
\begin{equation}\label{5.1}
R_{2k+1,n}(z)= \sum_{j=0}^{k-1}(n^{2j+2}-1)(n^{2k-2j}-1)\frac{B_{2j+2}B_{2k-2j}}{{(2j+2)!(2k-2j)!}}(nz)^{2k-2j}.
\end{equation} 
Replacing $z$ by $\frac{z}{n}$, we have
\begin{align}\label{5.2}
R_{2k+1,n}\left(\frac{z}{n}\right)&= z^2\sum_{j=0}^{k-1}(n^{2j+2}-1)(n^{2k-2j}-1)\frac{B_{2j+2}B_{2k-2j}}{{(2j+2)!(2k-2j)!}} z^{2k-2-2j}\nonumber\\
&=z^2H(z),
\end{align}
where $H(z)$ is defined as
\begin{align}\label{5.3}
H(z):=\sum_{j=0}^{k-1}(n^{2j+2}-1)(n^{2k-2j}-1)\frac{B_{2j+2}B_{2k-2j}}{{(2j+2)!(2k-2j)!}}z^{2k-2-2j}.
\end{align}
It is given that the coefficient of $z^{2k-2-2j}$ in \eqref{5.3} as 
\begin{align}\label{Aj}
A_{j}=(n^{2j+2}-1)(n^{2k-2j}-1)\frac{B_{2j+2}B_{2k-2j}}{{(2j+2)!(2k-2j)!}}.
\end{align}
Now we show that the polynomial $H(z)$ satisfies Theorem \ref{thm4.2}, that is,  the coefficients of  $H(z)$ satisfy the following inequality:
\begin{align}\label{5.4}
|A_{k-1}|\geq \inf_{\substack{c,d\in \mathbb{C}\\
{|d|=1}}}\sum_{j=0}^{k-1}|cA_j-d^{k-1-j}A_{k-1}|.  
\end{align} 
In particular,  considering $d=1$, we only need to show 
\begin{equation}\label{0.3}
|A_{k-1}|\geq\sum_{j=0}^{k-1}|cA_j-A_{k-1}|.
\end{equation}
Using the lower bound for Bernoulli numbers i.e.,  Lemma \ref{bound_Bernoulli} in \eqref{5.3}, we have
\begin{align}
|A_j|&=(n^{2j+2}-1)(n^{2k-2j}-1)\frac{|B_{2j+2}||B_{2k-2j}|}{{(2j+2)!(2k-2j)!}}\nonumber\\
&> \frac{4(n^{2j+2}-1)(n^{2k-2j}-1)}{(2\pi)^{2k+2}}. \label{lower bound Aj}
\end{align}
Similarly, using the upper bound for $B_{2k}$, we get
\begin{align}
|A_{k-1}|&=(n^{2k}-1)(n^2-1)\frac{|B_{2k}||B_2|}{(2k)!2!}\nonumber\\
&<(n^{2k}-1)(n^2-1)\frac{2.(2k)!}{(2\pi)^{2k}(1-2^{1-2k})(2k)!}\frac{2.2!}{(2\pi)^2(1-2^{1-2})2!}\nonumber\\
&=\frac{8(n^{2k}-1)(n^2-1)}{(2\pi)^{2k+2}(1-2^{1-2k})}. \label{upper bound A_{k-1}}
\end{align}
Note that the sign of $B_{2n}$ is $(-1)^{n+1}$,  which implies that the sign of $A_j$ will be $(-1)^{k+1}, \forall ~j$.  
Therefore,
$$
|cA_j-A_{k-1}|=(-1)^{k+1}(cA_j-A_{k-1}).
$$
Now we shall try to find a value of positive $c$ for which
\begin{align}\label{eq2}
&|cA_j-A_{k-1}|=(-1)^{k+1}(cA_j-A_{k-1})>0\nonumber\\
& \Leftrightarrow \quad c|A_j|-|A_{k-1}|>0.  
\end{align}
For any $n\in\mathbb{N}$,  it is well-known that
\begin{align*}
B_n\left(\frac{1}{2}\right)=(2^{1-n}-1)B_n,  \quad B_n(0)=B_n,  ~\forall n \neq 1.   
\end{align*}
We use these properties of Bernoulli numbers to see that
\begin{align}\label{5.9}
B_{2j+2}\left(\frac{1}{2}\right)-B_{2j+2}(0)& = B_{2j+2}(2^{1-2j-2}-1)-B_{2j+2}= B_{2j+2}(2^{1-2j-2}-2)\nonumber\\
\Rightarrow B_{2j+2}&=\frac{B_{2j+2}(\frac{1}{2})-B_{2j+2}(0)}{2(2^{-2j-2}-1)}.
\end{align}
Similarly,  one can show that
\begin{equation}\label{5.10}
 B_{2k-2j}=\frac{B_{2k-2j}(\frac{1}{2})-B_{2k-2j}(0)}{2(2^{-2k+2j}-1)}.
\end{equation}
Now employing the above identities \eqref{5.9} and \eqref{5.10} in  \eqref{Aj}, we obtain
\begin{align}\label{5.11}
A_{j}=& (n^{2j+2}-1)(n^{2k-2j}-1)\frac{B_{2j+2}B_{2k-2j}}{{(2j+2)!(2k-2j)!}}  \nonumber\\
=& {{2k+2}\choose{2j+2}}(n^{2j+2}-1)(n^{2k-2j}-1) \frac{B_{2j+2}B_{2k-2j}}{(2k+2)!}  \nonumber\\
=& {{2k+2}\choose{2j+2}}\frac{(n^{2j+2}-1)(n^{2k-2j}-1)}{(2k+2)!} \left(  \frac{B_{2j+2}(\frac{1}{2})-B_{2j+2}(0)}{2(2^{-2j-2}-1)} \right) \left( \frac{B_{2k-2j}(\frac{1}{2})-B_{2k-2j}(0)}{2(2^{-2k+2j}-1)} \right).
\end{align}
Recall that the sum on the right hand side of \eqref{0.3} is
$$\sum_{j=0}^{k-1}|cA_j-A_{k-1}|=(-1)^{k+1}\sum_{j=0}^{k-1}\left(cA_j- A_{k-1}\right).$$
Using \eqref{5.11}, we can write
\begin{align}\label{sumcAj}
&(-1)^{k+1}\sum_{j=0}^{k-1}cA_j= \sum_{j=0}^{k-1}c|A_j|\nonumber\\
& = \sum_{j=0}^{k-1}c{{2k+2}\choose{2j+2}}\frac{(n^{2j+2}-1)(n^{2k-2j}-1)}{(2k+2)!}  \Bigg|   \left( \frac{B_{2j+2}(\frac{1}{2})-B_{2j+2}(0)}{2(2^{-2j-2}-1)} \right) \left( \frac{B_{2k-2j}(\frac{1}{2})-B_{2k-2j}(0)}{2(2^{-2k+2j}-1)} \right)\Bigg| \nonumber\\
& =  \sum_{j=0}^{k-1}D(j){{2k+2}\choose{2j+2}} \Bigg| \left(B_{2j+2}\left(\frac{1}{2}\right)-B_{2j+2}(0)\right)\left(B_{2k-2j}\left(\frac{1}{2}\right)-B_{2k-2j}(0)\right)\Bigg|,
\end{align}
where
\begin{align}\label{5.12}
D(j)&=\frac{c}{4(2k+2)!}\frac{(n^{2j+2}-1)(n^{2k-2j}-1)}{(2^{-2j-2}-1)(2^{-2k+2j}-1)}=\frac{2^{2k}c}{(2k+2)!}\frac{(n^{2j+2}-1)}{(2^{2j+2}-1)}\frac{(n^{2k-2j}-1)}{(2^{2k-2j}-1)}.
\end{align}
Using  Lemma \ref{lemma3.3},  one can find the following upper bound for $D(j)$:
\begin{align}\label{Dj}
D(j)=\frac{2^{2k}c}{(2k+2)!}g_n(2j)< \frac{2^{2k}c}{(2k+2)!}\frac{(n^{k+1}-1)^2}{3(2^{2k}-1)}.
\end{align}
Utilize \eqref{Dj} in \eqref{sumcAj} to see that
\begin{align}\label{5.13}
\sum_{j=0}^{k-1}c|A_j |& <\frac{2^{2k}c}{(2k+2)!}\frac{(n^{k+1}-1)^2}{3(2^{2k}-1)}\sum_{j=0}^{k-1}{{2k+2}\choose{2j+2}}\nonumber\\
& \Bigg| \left(B_{2j+2}\left(\frac{1}{2}\right)-B_{2j+2}(0)\right)\left(B_{2k-2j}\left(\frac{1}{2}\right)-B_{2k-2j}(0)\right) \Bigg|.
\end{align}
Now we shall use Lemma \ref{lemma1.1} to calculate the value of the sum given in \eqref{5.13}. That is,
\begin{align}\label{5.15}
&\sum_{j=0}^{k-1}{{2k+2}\choose{2j+2}}\Bigg| \left(B_{2j+2}\left(\frac{1}{2}\right)-B_{2j+2}(0)\right)\left(B_{2k-2j}\left(\frac{1}{2}\right)-B_{2k-2j}(0)\right) \Bigg| \nonumber\\
&=(-1)^{k-1} \sum_{j=0}^{k-1}{{2k+2}\choose{2j+2}}  \Bigg[ B_{2j+2}\left(\frac{1}{2}\right)B_{2k-2j}\left(\frac{1}{2}\right) -B_{2j+2}\left(\frac{1}{2}\right)B_{2k-2j}(0)\nonumber\\
& -B_{2j+2}(0)B_{2k-2j}\left(\frac{1}{2}\right)+B_{2j+2}(0)B_{2k-2j}(0)\Bigg]\nonumber\\
&=(-1)^k 4(2k+1)(1-2^{-2k-2})B_{2k+2}.
\end{align}
Employing \eqref{5.15} in \eqref{5.13}, we get
\begin{equation}\label{upp}
\sum_{j=0}^{k-1}c |A_j| < \frac{2^{2k+2}c}{(2k+2)!}\frac{(n^{k+1}-1)^2}{3(2^{2k}-1)}(2k+1)(1-2^{-2k-2}) |B_{2k+2}|.
\end{equation}
Our main aim is to find a positive $c$ such that \eqref{0.3} holds.  From \eqref{0.3}, we have
\begin{align}\label{5.16}
&|A_{k-1}|>\sum_{j=0}^{k-1}|cA_j-A_{k-1}|\nonumber\\
& \Leftrightarrow (-1)^{k+1}A_{k-1}>(-1)^{k+1}  \sum_{j=0}^{k-1}(cA_j-A_{k-1})\nonumber\\
&\Leftrightarrow  (1+k)(-1)^{k+1}A_{k-1}>(-1)^{k+1}\sum_{j=0}^{k-1}cA_j \nonumber \\
&\Leftrightarrow (1+k) |A_{k-1}|>\sum_{j=0}^{k-1}c|A_j|.  
\end{align}
To prove \eqref{5.16},  we shall show that the left hand side expression of \eqref{5.16} is in fact bigger than an upper bound of the sum $\sum_{j=0}^{k-1}c|A_j|$, which we have already found in \eqref{upp}. So we need to find a positive $c$ such that the following inequality holds:
\begin{align}
&(1+k) |A_{k-1}|> \frac{c}{(2k+2)!}\frac{(n^{k+1}-1)^2}{3(2^{2k}-1)}(2k+1)(2^{2k+2}-1) |B_{2k+2}|.
\end{align}
This implies that $c$ should satisfy the below inequality:
\begin{align}\label{inequa_for c}
\frac{3 (1+k) (2^{2k}-1) (2k+2)! |A_{k-1}|}{(2k+1)(n^{k+1}-1)^2 (2^{2k+2}-1) |B_{2k+2}|} > c.
\end{align}
Moreover,  using the lower bound  \eqref{lower bound Aj} of $|A_{k-1}|$ and upper bound \eqref{bound_Bernoulli} for $|B_{2k+2}|$,  one can show that 
\begin{align*}
\frac{3 (1+k) (2^{2k}-1) (2k+2)! |A_{k-1}|}{(2k+1)(n^{k+1}-1)^2 (2^{2k+2}-1) |B_{2k+2}|}  > c_{n,k}
\end{align*}
where
\begin{align*}
c_{n,k}:= \frac{(1+k)(n^{2k}-1)(n^2-1)(2^{2k}-1)(2^2-1) (2^{2k+1}-1)}{2^{2k}(2k+1)(n^{k+1}-1)^2(2^{2k+2}-1)}.
\end{align*}
Note the constant $c_{n,k}$ is a positive constant.  Thus,  for any positive $c$ such that $c < c_{n,k}$ will satisfy the inequality \eqref{inequa_for c} and hence  the equation \eqref{5.16}.  
This proves the inequality \eqref{0.3} for the polynomial $H(z)$ and consequently the result follows.
\end{proof}

Now we are ready for proving the main theorem.
\begin{proof}[Theorem {\rm \ref{main theorem}}][]
We are interested to show that all non-real roots of $R_{2k+1,n}(z)$ lie on the circle $z=1/n$,  which is equivalent to show that all non-real roots of $R_{2k+1, n}\left( \frac{z}{n} \right)$ lie on the unit circle $|z|=1$.  From the equation \eqref{5.2},  we know that 
\begin{align}\label{5.18}
R_{2k+1,n}\left(\frac{z}{n}\right)&= z^2\sum_{j=0}^{k-1}(n^{2j+2}-1)(n^{2k-2j}-1)\frac{B_{2j+2}B_{2k-2j}}{{(2j+2)!(2k-2j)!}} z^{2k-2-2j}\nonumber\\
&=z^2H(z),
\end{align}
where $H(z)$ is defined as
\begin{align*}
H(z)=\sum_{j=0}^{k-1} A_j  z^{2k-2-2j},  
\end{align*}
and
\begin{align*}
A_{j}=(n^{2j+2}-1)(n^{2k-2j}-1)\frac{B_{2j+2}B_{2k-2j}}{{(2j+2)!(2k-2j)!}}.
\end{align*}
We can easily check that $A_j$ and $A_{k-1-j}$ are same, which indicates that the polynomial $R_{2k+1,n}\left( \frac{z}{n}  \right)$ is reciprocal and so $R_{2k+1,n}(z)$. 
Now from \eqref{5.18}, we can easily see that $R_{2k+1,  n}\left(\frac{z}{n}\right)$ has degree $2k$ containing only even powers of $z$ with $z^2$ being the least power of $z$. Therefore,  $R_{2k+1,n}\left(\frac{z}{n}\right)$ has a real zero at $ z=0$ of multiplicity $2$.  Utilizing Proposition \ref{Schinzel's criteria for Rama_type},  we can see that $H(z)$ satisfies Theorem \ref{thm4.2}, so we can conclude that $H(z)$ has all its roots on the unit circle.  Moreover,  as we have proved that the inequality \eqref{5.16} is strict so all the zeros of $H(z)$ are simple. Thus, the polynomial $R_{2k+1,n}\left(\frac{z}{n}\right)$ has all the zeros on the unit circle except $z=0$.
Next,  we show that there is no real root except zero.
Let us suppose that the polynomial $R_{2k+1,n}\left(\frac{z}{n}\right)$ has a  real root except $z=0$, then it must be on the unit circle.  Consequently,  the possibility of the non-trivial real root will be either $ z=1 ~\text{or} ~z=-1$ or both.  So, putting $z=1$ in \eqref{5.3}, we get
\begin{align}\label{H(1)}
&H(1)=0,\nonumber\\
\Rightarrow & \sum_{j=0}^{k-1}(n^{2j+2}-1)(n^{2k-2j}-1)\frac{B_{2j+2}B_{2k-2j}}{{(2j+2)!(2k-2j)!}}=0,\nonumber\\
\Rightarrow & \sum_{j=0}^{k-1}A_j=0.
\end{align}
We have already seen that the sign of $A_j$ is $(-1)^{k+1}$ for all $j$.  Note that all even Bernoulli numbers are never zero. Therefore, depending on the parity of $k$ the sum $\sum_{j=0}^{k-1}A_j$ is either positive or negative.  This is a contradiction to \eqref{H(1)}.  This completes the proof that $z=1$ is not a root of $H(z)$. Again,  as $H(z)$ is an even function, so by the same argument we can conclude that $z=-1$ is not a root of $H(z)$ as well. Hence, all roots of $H(z)$ are complex and simple, and lie on the unit circle. This completes the proof.
\end{proof}

\section{Concluding Remarks}
Motivated from the Ramanujan's formula \eqref{1.2} as well as Grosswald's identity \eqref{1.3},  Gun,  Murty,  and Rath \cite{item2} defined Ramanujan polynomial as
 \begin{equation}\label{Rama_poly}
R_{2k+1}(z)=\sum_{j=0}^{k+1}z^{2k+2-2j}\frac{B_{2j}}{(2j)!}\frac{B_{2k+2-2j}}{(2k+2-2j)!}.
\end{equation}
Around the same time,  Murty,  Smith,  and Wang \cite{MSW11} showed that all the non-real roots of $R_{2k+1}(z)$ lie on the unit circle.  In this paper,  we studied the following Ramanujan-type polynomial: 
\begin{equation}\label{Rama-type-polynomial}
R_{2k+1,n}(z):= \sum_{j=1}^k(n^{2j}-1)(n^{2k+2-2j}-1)\frac{B_{2j}B_{2k+2-2j}}{{(2j)!(2k+2-2j)!}}(nz)^{2k+2-2j},  
\end{equation}
We showed that it has only one real root is at $z=0$ of multiplicity $2$.  Moreover,  the non-real roots of $R_{2k+1,n}(z)$  are simple and all of them lie on the circle $|z|=\frac{1}{n}$,  which says that all the non-real roots of $R_{2k+1,n}\left( \frac{z}{n} \right)$ lie on the unit circle.  
Ramanujan's formula \eqref{1.2} has been generalized by many mathematicians,  over the years.  Recently,  Dixit and Gupta \cite[Theorem 2.1] {DG19} found an interesting analogue of Ramanujan's formula \eqref{1.2} for squares of odd zeta values and in the process they encountered the following finite sum involving Bernoulli numbers.  For any $k \in \mathbb{N}$,  
\begin{align}\label{Dixit-Gupta-sum}
\sum_{j=0}^{k+1} (-1)^j \frac{B_{2j}^2}{{(2j)!}^2}\frac{B_{2k+2-2j}^2}{{(2k+2-2j)!}^2} \alpha^{2j} \beta^{2k+2-2j}.
\end{align}
Very recently,  Banerjee and Sahani \cite[Theorem 1.1]{BS22} generalized the work of Dixit and Gupta and found a formula for higher power of odd zeta values.  While obtaining this generalization,  they found the following finite sum involving higher powers of Bernoulli numbers.  For any two fixed natural numbers $k$ and $\ell$,
\begin{align}\label{Banerjee-Sahani-sum}
\sum_{j=0}^{k+1} (-1)^j \frac{B_{2j}^\ell}{{(2j)!}^\ell}\frac{B_{2k+2-2j}^\ell}{{(2k+2-2j)!}^\ell} \alpha^{\ell j} \beta^{\ell(k+1-j)}.
\end{align}
It is interesting to note that,  in both of the equations \eqref{Dixit-Gupta-sum} and \eqref{Banerjee-Sahani-sum},  $\alpha$ and $\beta$ are positive and satisfy $\alpha \beta = \pi^2$,  which is exactly the same  condition also required in Ramanujan's formula \eqref{1.2}.  This observation motivated us to define a one-variable generalization of Ramanujan polynomial.  
Let $k$ and $\ell$ be two fixed natural numbers,  then we define
\begin{align}
R_{2k+1}^{(\ell)}(z):= 
\sum_{j=0}^{k+1} (-1)^{(\ell+1)j} \frac{B_{2j}^\ell}{{(2j)!}^\ell}\frac{B_{2k+2-2j}^\ell}{{(2k+2-2j)!}^\ell} z^{2\ell j}.  
\end{align}
For $\ell=1$,  this polynomial coincides with the Ramanujan polynomial \eqref{Rama_poly}.  For simplicity,  replace $z^{2 \ell}$ by $Z$  to write
\begin{align}
R_{2k+1}^{(\ell)}(Z):= 
\sum_{j=0}^{k+1} (-1)^{(\ell+1)j} \frac{B_{2j}^\ell}{{(2j)!}^\ell}\frac{B_{2k+2-2j}^\ell}{{(2k+2-2j)!}^\ell} Z^{j}.  
\end{align}
We can easily verify that this is a reciprocal polynomial.  Based on numerical evidences checked in Mathematica software,  we conjecture the following statement. 
\begin{conjecture}\label{conjecture1}
All the non-real roots of $R_{2k+1}^{(\ell)}(Z)$ are simple and lie on the unit circle $|Z|=1$. 
\end{conjecture}

%
%

Ramanujan's formula for $\zeta(2k+1)$ has been generalized by many mathematicians for different $L$-functions.  For examples,  readers are encouraged to see \cite{BGK23,  Berndt75,  Berndt77,  BS17,   bradley,  CCVW21,  DGKM20,  DM20,  GuptaMaji} and references therein.  Interestingly,  all of these generalizations consist of a finite sum involving Bernolli numbers or Euler numbers or some special values of $L$-functions. 
This suggests that we can always find a generalized Ramanujan polynomial 
from any analogue of Ramanujan's formula for $\zeta(2k+1)$ and that might be interest of study. 
 In this spirit,  Berndt and Straub studied a character analogue of Ramanujan polynomial and gave a more general conjecture \cite[Conjecture 7.3]{BS16},  which is different from our Conjecture \ref{conjecture1}.

\section{Acknowledgement}

The first author wants to thank  SERB for the MATRICS Grant MTR/2022/000545.  This article is based on the M.Sc.  thesis work of the second author under the guidance of the first author.  Both authors sincerely thank Bhaskaracharya Mathematics Laboratory and  Brahmagupta Mathematics Library of the Department of Mathematics at IIT Indore,  supported by DST FIST Project (File No.: SR/FST/MS-I/2018/26).

\end{document}